\documentclass[a4paper,12pt]{amsart}

\usepackage[a4paper,left=3cm,right=3cm,top=3cm,bottom=2.65cm]{geometry}
\usepackage{enumerate, mathrsfs}
\usepackage{xcolor}

\usepackage{enumitem}

\usepackage[colorinlistoftodos]{todonotes}

\usepackage[foot]{amsaddr}

\usepackage{amsmath,amsfonts,amsthm,amssymb}
\usepackage[all]{xy}
\usepackage[T1]{fontenc}

\usepackage[colorlinks=true]{hyperref}
\usepackage{cleveref}
\usepackage{tikz}
\usetikzlibrary{shapes.geometric}
\usetikzlibrary{calc}
\usetikzlibrary{arrows}

\newcounter{intege}

\theoremstyle{plain}
\newtheorem{theorem}{Theorem}[section]

\newtheorem{corollary}[theorem]{Corollary}
\newtheorem{lemma}[theorem]{Lemma}

\theoremstyle{definition}

\newtheorem{conjecture}[theorem]{Conjecture}

\crefname{counterexample}{counterexample}{counterexamples}
\Crefname{counterexample}{Counterexample}{Counterexamples}
\crefname{conjecture}{conjecture}{conjectures}
\Crefname{conjecture}{Conjecture}{Conjectures}

\definecolor{lightgrey}{rgb}{0.7,0.7,0.7}

\definecolor{darkred}{rgb}{0.7,0,0} 
\newcommand{\darkred}{\color{darkred}} 

\newcommand{\C}{\mathcal C}

\newcommand{\family}{\mathcal F}

\DeclareMathOperator{\rank}{rank}
\DeclareMathOperator{\st}{st}

\newcommand{\defn}[1]{\emph{\darkred #1}} 


\title{Short proof of two cases of Chv\'atal's conjecture}

\date{\today}

\author[J.~Olarte]{Jorge Olarte}
\author[F.~Santos]{Francisco Santos}
\author[J.~Spreer]{Jonathan Spreer}

\address[J.~Olarte and J.~Spreer]
{
Institut f\"ur Mathematik, Freie Universit\"at Berlin, Germany
}
\email{olarte@zedat.fu-berlin.de}
\email{jonathan.spreer@fu-berlin.de}

\address[F.~Santos]
{
Department of Mathematics, Statistics and Computer Science, University of Cantabria, Spain
}
\email{francisco.santos@unican.es}

\thanks{The three authors are supported by the Einstein Foundation, Berlin. F. Santos is also supported by grant MTM2014-54207-P of the Spanish  Ministry of Science.}


\begin{document}

\begin{abstract}
  In 1974 Chv\'atal conjectured that no intersecting family $\family$ in a downset can be larger than the largest star.
  In the same year Kleitman and Magnanti proved the conjecture when $\family$ is contained in the union of two stars,
  and Sterboul when $\rank(\family)\le 3$. 
  We give short self-contained proofs of these two statements.
\end{abstract}

\maketitle

\section{Introduction}

A \defn{downset}, \defn{hereditary set}, \defn{independence system} or \defn{(abstract) simplicial complex} $\C$ is a family of subsets of some finite ground set closed under taking subsets. Using nomenclature from simplicial complexes we call \defn{faces}  the elements of $\C$ and \defn{vertices}, \defn{edges} and \defn{triangles}, respectively, the faces  of sizes $1$, $2$ and $3$.
The \defn{star} of a vertex $a$, written $\st_\C(a)$, is the family of all faces containing $a$. It is an example of an \defn{intersecting family} in $\C$, that is, a set of faces that pairwise intersect.

Chv\'atal's 45-year-old conjecture, inspired by the classical result of Erd\H{o}s, Ko and Rado~\cite{EKR1961} for the complete, uniform complex $\binom{[n]}{\le k}$, states that stars always achieve the maximal cardinality among intersecting families in $\C$:

\begin{conjecture}[Chv{\'a}tal~\cite{Chv1974}]
  \label{chvatal}
  Let $\family$ be an intersecting family in a simplicial complex~$\C$. Then, there exists a vertex $a$ in~$\C$ such that 
  $|\family| \le |\st_\C(a)|$.
\end{conjecture}

Note that there is no loss of generality in assuming that~$\C$ is the smallest downset containing $\family$ (in other words,~$\C$ is generated by $\family$). Throughout the article we assume this and give short proofs of the following old and very recent known cases. 

\begin{theorem}[Kleitman and Magnanti~\protect{\cite[Theorem 2]{Kle1974}}]
\label{thm:maintwo}
Let~$\family$ be an intersecting family contained in the union of two stars, $\family\subset \st(a) \cup \st(b)$.
Then $ |\family| \le \max( |\st(a)|, |\st(b)| )$.
\end{theorem}

\begin{theorem}[Sterboul \protect{\cite[Theorem 2]{Ste1974}}]
\label{thm:main_intro}
Chv\'atal's conjecture holds if all elements of $\family$ have size three or less.
\end{theorem} 

Equivalently, this result settles Chv\'atal's conjecture for rank at most three.
It was recently reproven by Czabarka, Hurlbert and Kamat~\cite[Theorem 1.4]{czabarka2017chv}.
We thank G. Hurlbert for pointing us towards reference \cite{Ste1974}.

Our proofs are inspired by our recent work with Stump on a related EKR problem~\cite{OSSS2018}.

\section{Intersecting families contained in two stars}
\label{sec:sizetwo}


\begin{lemma}
  \label{lemma:flip}
  Let~$\family$ be an intersecting family in~$\C$. 
  Let $a,b,v$ be three vertices of~$\C$ and assume that every $B\in \family$
  with $v\in B$ intersects $\{a, b\}$.
    Define
  \begin{align*}
    R_a(v) &:= \{ B\in \family : a,v\in B, b\not\in B, B\setminus v\not\in \family\},\\
    R_b(v) &:= \{ B\in \family : b,v\in B, a\not\in B, B\setminus v\not\in \family\}.
   \end{align*}
  Then,
$ 
    \family' := \family \setminus R_b(v) \cup \{B\setminus v : B \in R_a(v)\}
$ 
  is also an intersecting family.
\end{lemma}

\begin{proof}
  All sets in $\{B\setminus v : B \in R_a(v)\}$ intersect one another since they all contain $a$. 
  We thus only need to show that  every $B_1\in \family \setminus R_b(v)$ intersects every $B_2\in R_a(v)$  in an element different from $v$. 
  If $v\not\in B_1$ this is obvious since $B_1$ and $B_2$ are both in $\family$ and thus they meet. 
  If $a\in B_1$  this is obvious too, since then $a \in B_1\cap B_2$. 
  Hence, assume $B_1$ contains $v$ but not $a$. 
  Our hypotheses imply that $b \in B_1$  and since $B_1\not\in R_b(v)$ we have that $B_1\setminus {v} \in \family$. 
  Thus, $(B_1\setminus {v})\cap B_2$ indeed meet.
\end{proof}

\Cref{thm:maintwo} is a direct consequence of the following statement.

\begin{corollary}
\label{coro:sizetwo}
Let~$\family \subset \C$ be an intersecting family such that $\family\subset \st(a) \cup \st(b)$ and neither $\st(a)$ nor $\st(b)$ contains $\family$. Then there exists an intersecting family $\family'\subset \C$, $\family'\subset \st(a) \cup \st(b)$, such that either $|\family'| > |\family|$ or $|\family'| = |\family|$ but then $\family'$ has smaller average size of elements than~$\family$.
\end{corollary}

\begin{proof}
We first claim that there exists a vertex $v$ in~$\C$ such that (at least) one of the sets $R_a(v)$ and $R_b(v)$ of the previous lemma is not empty. For this,  let $B$ be a minimal face in $\family$ containing $a$ but not $b$ 
(it exists, or else the condition $\family\subset \st(a) \cup \st(b)$ implies $\family \subset \st(b)$). 
If $B=\{a\}$ then $\family \subset \st(a)$. 
If $B\ne\{a\}$ then for each $ v\in B\setminus a$ we have $R_a(v)\ne \emptyset$.

Assume that either $|R_a(v)| > |R_b(v)|$ or
$|R_a(v)| = |R_b(v)|$ and $|R_a(v)|$ has average size of sets smaller or equal than $|R_b(v)|$. This is no loss of generality since $|R_a(v)| \le |R_b(v)|$ implies $R_b(v)$ is not empty and we can exchange the roles of $a$ and $b$.

Hence, we have $|\family'| = |\family| - |R_b(v)| + |R_a(v)| \ge |\family|$ with equality only if $|R_a(v)| = |R_b(v)|$. In this case, since  
$|R_a(v)|$ has average size of sets smaller or equal than  $|R_b(v)|$ and we substitute the sets of $R_b(v)$ with sets of size smaller than those of $R_a(v)$, the average size of sets in $\family'$ is smaller than in $\family$.
\end{proof}

\section{Intersecting families of rank three}
\label{sec:triangles}

To simplify notation, in what follows we omit braces when referring to a subset of the ground set and write, e.~g., $abc$ instead of $\{a,b,c\}$.
In part (1) of the following statement, given a triangle $abc \in \family$ we say that a second triangle $\tau\in \family$ is \defn{dangling} from $abc$ at one of the vertices $x\in abc$ if $\tau\cap abc=x$.

\begin{lemma}
\label{lemma:all}
Let $\family$ be an intersecting family consisting only of triangles. If any of the following conditions is satisfied, then there exists an intesecting family of size at least $|\family|$ containing an edge or vertex:
\begin{enumerate}
  \item Some triangle in $\family$ has one or no triangles dangling at some vertex;
  \item No two triangles in $\family$ share an edge;
  \item The graph of the complex generated by $\family$ is not complete.
\end{enumerate}

\end{lemma}


\begin{proof}
Throughout the proof, let $abc$ be a triangle in $\family$.

For part (1), if for some vertex, say $a$, there is only one triangle $\tau\in\family$ dangling at $a$, let $\family' = \family \setminus \{\tau \} \cup \{bc\}$. If there is none, just add $bc$ to $\family$.

\smallskip

For part (2), assume without loss of generality that among the triangles of $\family$ there are at least as many containing $a$ than $b$ or $c$. 
Let $\family'$ consist of the triangle $abc$ plus all other triangles $axy\in\family$ 
together with their edges $ax$ and $ay$. Then $\family'\subsetneq \st_\C(a)$ and $|\family'|\ge |\family|$ since all edges $ax$ and $ay$ are distinct. 

\smallskip

For part (3), let $c$ and $v$ be vertices not spanning an edge. Let $abc \in \family$ be a triangle containing $c$, and let $S_v=\left \{x\in abc : \exists \, y \text{ with } vxy \in \family \right \}$. By the hypothesis, $S_v \subset ab$. The assumption that~$\C$ is generated by the faces of $\family$ implies that $S_v \neq \emptyset$, so we assume $a \subset S_v$.
If $S_v = a$ then we add $ay$ to $\family$ for each $avy \in \family$.
Hence, assume for the rest that $S_v = ab$. Note that every element of $\family$ containing $v$ must contain either $a$ or $b$ since $\family$ is intersecting. In particular, we can apply \Cref{lemma:flip}. If one of $R_a(v)$ and $R_b(v)$ is non-empty this yields an intersecting family $\family'$, $|\family'| \ge |\family|$, containing edges. If both $R_a(v)$ and $R_b(v)$ are empty, the only element of $\family$ containing $v$ is $abv$ and we can add $ab$ to $\family$. 
\end{proof}

%
%




\begin{proof}[Proof of \Cref{thm:main_intro}]
Chv\'atal's conjecture holds when $\family$ contains a vertex (trivial) or an edge (\Cref{thm:maintwo}; note that if $ab \in \family$, then trivially $\family\subset \st(a) \cup \st(b)$). Thus, we can assume that $\family$ consists entirely of triangles and, by \Cref{lemma:all}, that it does not satisfy any of the three conditions listed in that lemma.

In particular, by part (2) of the lemma, $\family$ contains two triangles $abc$ and $abx$ sharing an edge. Observe that all triangles dangling from $abc$ at $c$ must contain $x$ since a triangle dangling at $c$ and not containing $x$ does not intersect $abx$. Moreover:

  \begin{itemize}[leftmargin=.7cm,noitemsep,topsep=0pt]
    \item {\em There are exactly two such triangles, say $cxy$ and $cxz$.} There are at least two by part (1) of \Cref{lemma:all}. If there is a third triangle $cxv$ dangling at $c$, then every triangle in $\family$ must intersect $cx$ (otherwise it must contain $v$, $y$ and $z$ and intersect $abc$, a contradiction). Hence, we can add $cx$ to $\family$ and apply \Cref{thm:maintwo}.
    \item {\em The only vertices of $\C$ are $a,b,c,x,y,z$.} Assume there exists another vertex $v \in \C$. By part (3) of \Cref{lemma:all} the edge $cv$ is contained in some triangle $\tau \in \family$. By the previous item, $\tau$ is not dangling from $abc$ at $c$ so without loss of generality $\tau = acv$.  Now, every triangle $\sigma \in \family$ dangling at $b$ must contain both $v$ and either $x$ or both $y$ and $z$. Since the latter is impossible, $bvx$ is the only possible triangle dangling at $b$, contradicting part (1) of \Cref{lemma:all}.
  \end{itemize} 
  
 Once we know there are exactly six vertices, observe that at most half of the $\binom63=20$ triangles on six vertices, one from each complementary pair, can be in $\family$, so $|\family|\le 10$. But the above implies that $\st_\C(c)$ contains at least the following 10 faces: the three triangles $abc$, $cxy$, $cxz$ plus at least another triangle dangling from $cxy$ at $c$, the five edges $ca$, $cb$, $cx$, $cy$, $cz$, and $c$ itself.
\end{proof}


\bibliographystyle{plain}
\bibliography{bibliography}

\begin{thebibliography}{1}

\bibitem{Chv1974}
V.~Chv{\'a}tal.
\newblock Intersecting families of edges in hypergraphs having the hereditary
  property.
\newblock In {\em Hypergraph Seminar: Ohio State University 1972}, vol. 411
  of {\em Lect. Notes Math.}, pp.~61--66. Springer Berlin Heidelberg, 1974.

\bibitem{czabarka2017chv}
E.~Czabarka, G.~Hurlbert, and V.~Kamat.
\newblock Chv{\'a}tal's conjecture for downsets of small rank, 2017.
\newblock Preprint, \href{https://arxiv.org/abs/1703.00494}{arXiv:1703.00494}.

\bibitem{EKR1961}
P.~Erd{\H o}s, C.~Ko, and R.~Rado.
\newblock Intersection theorems for systems of finite sets.
\newblock {\em Quart. J. Math. Oxford Ser. (2)}, 12:313--320, 1961.

\bibitem{Kle1974}
D.~J. Kleitman and T.~L. Magnanti.
\newblock On the number of latent subsets of intersecting collections.
\newblock {\em J. Combinatorial Theory Ser. A}, 16:215--220, 1974.

\bibitem{OSSS2018}
Jorge~A. Olarte, Francisco Santos, Jonathan Spreer, and Christian Stump.
\newblock The ekr property for flag pure simplicial complexes without boundary,
  2017.
\newblock Preprint, \href{https://arxiv.org/abs/1710.02518}{arXiv:1710.02518}.

\bibitem{Ste1974}
F.~Sterboul.
\newblock Sur une conjecture de {V}. {C}hv\'atal.
\newblock In {\em Hypergraph Seminar: Ohio State University 1972}, vol. 411
  of {\em Lect. Notes Math.}, pp.~152--164. Springer Berlin Heidelberg, 1974.

\end{thebibliography}

\end{document}